\documentclass[a4paper,11pt]{article}
\usepackage[ansinew]{inputenc}
\usepackage[english]{babel}
\usepackage{amssymb}
\usepackage{amsmath}
\usepackage{amsthm}
\usepackage[pdftex]{color, graphicx}
\usepackage{mathrsfs}
\usepackage{verbatim}
\usepackage{ytableau}
\usepackage{stmaryrd}

\theoremstyle{definition} \newtheorem{defi}{Definition}[subsection]

\theoremstyle{plain} \newtheorem{thm}[defi]{Theorem}

\theoremstyle{plain} \newtheorem{prop}[defi]{Proposition}

\theoremstyle{plain} \newtheorem{lm}[defi]{Lemma}

\theoremstyle{plain} \newtheorem{cor}[defi]{Corollary}

\theoremstyle{definition} \newtheorem{ex}[defi]{Example}

\theoremstyle{definition} \newtheorem{rem}[defi]{Remark}

\newcommand{\Z}[0]{\mathbb{Z}}

\newcommand{\St}{\operatorname{St}}
\newcommand{\soc}{\operatorname{soc}}
\newcommand{\id}{\operatorname{id}}
\newcommand{\Hom}{\operatorname{Hom}}
\newcommand{\Ext}{\operatorname{Ext}}
\newcommand{\F}[0]{\mathbb{F}}
\newcommand{\G}{\mathcal{G}}
\newcommand{\ind}{\operatorname{ind}}
\newcommand{\Gq}{G(\F_q)}
\newcommand{\ZX}{\mathbb{Z}[X]^W}
\newcommand{\Q}{\mathcal{Q}}
\newcommand{\lb}{\llbracket}
\newcommand{\rb}{\rrbracket}
\newcommand{\wh}[1]{\widehat{#1}}

\author{Tobias Kildetoft\footnote{Supported in part by  QGM (Centre for Quantum Geometry of Moduli Spaces) funded by
the Danish National Research Foundation and in part by  Knut and Alice Wallenbergs Foundation}\\ Uppsala University\\ tobias.kildetoft@gmail.com}
\title{Decomposition of tensor products involving a Steinberg module}
\begin{document}

\maketitle


\begin{abstract}

We study the decomposition of tensor products between a Steinberg module and a costandard module, both as a module for the algebraic group $G$ and when restricted to either a Frobenius kernel $G_r$ or a finite Chevalley group $\Gq$. In all three cases, we give formulas reducing this to standard character data for $G$.

Along the way, we use a bilinear form on the characters of finite dimensional $G$-modules to give formulas for the dimension of homomorphism spaces between certain $G$-modules when restricted to either $G_r$ or $\Gq$. Further, this form allows us to give a new proof of the reciprocity between tilting modules and simple modules for $G$ which has slightly weaker assumptions than earlier such proofs. Finally, we prove that in a suitable formulation, this reciprocity is equivalent to Donkin's tilting conjecture.\\


\end{abstract}

\section{Introduction}

Let $G$ be a semisimple, connected, simply connected algebraic group scheme over an algebraically closed field $k$ of characteristic $p>0$. When studying representations of $G$, a natural starting point are the costandard modules $\nabla(\lambda) = \ind_B^G(\lambda)$ where $B$ is a Borel subgroup and $\lambda$ is a dominant weight. These modules have characters given by Weyl's character formula and simple socles $L(\lambda)$ which exhaust the simple $G$-modules.

A series of costandard modules of special interest are the Steinberg modules $\St_r = \nabla((p^r-1)\rho)$ where $\rho$ is the Weyl weight. These are simple and self-dual.

In this paper we will study modules of the form $\St_r\otimes\nabla(\lambda)$, in particular questions of how these decompose into indecomposable summands. Under certain conditions on $\lambda$ these are tilting, meaning that both the module and its dual have filtrations with subfactors isomorphic to costandard modules (i.e. they both have good filtrations). Thus, the module can be written as a direct sum of indecomposable tilting modules $T(\nu)$ for suitable weights $\nu$. Further, in many cases, the decomposition is completely determined by the socle of the module, and we produce formulas for multiplicities in this socle which only rely on standard character data for $G$.

It will turn out to be convenient to also study modules of the more general form $T((p^r-1)\rho + \mu)\otimes\nabla(\lambda)$ for some dominant weight $\mu$, since this allows for induction on $r$ in certain cases. Similarly to the above, these are also tilting under the same conditions on $\lambda$, and if we further put some restrictions on $\mu$ the decomposition is again determined by the socle.

The motivation for studying specifically modules of the form $\St_r\otimes\nabla(\lambda)$ is that the characters of these are known, so if one can determine the decomposition, there is a hope that one might also obtain information about the characters of the indecomposable summands. As mentioned, these summands will in some cases be indecomposable tilting modules, and the characters of these are still not very well understood outside type $A$, where they have been shown to be given in terms of the $p$-canonical basis in \cite{richewilliamson15} where a general conjecture for their characters is also proposed.

The main way we are able to reduce these question to standard character data for $G$ is to show that certain tensor products involved have good filtrations, extending previous work by Andersen, Nakano and the author. Once this is shown, we can apply a result of Donkin which relates the number of times a costandard module occurs in a good filtration to the dimensions of suitable spaces of homomorphism of $G$-modules.

Since we are working in positive characteristic, there is a Frobenius map $F: G\to G$, and both the kernel $G_r$ and fixed points $\Gq$ of the $r$'th iterate of this are of great interest. The first is called the $r$'th Frobenius kernel and the second is a finite Chevalley group which can be identified with the points of $G$ over the field with $q = p^r$ elements, hence the notation.

We also study the restriction of the above mentioned tensor products to these subgroups. In both cases, the restriction is projective and injective and thus the problem once again becomes to describe socles and hence dimensions of certain homomorphism spaces for these subgroups.

More generally, we will provide formulas for the dimension of the space of homomorphisms between a tilting module for $G$ of the form $T((p^r-1)\rho + \lambda)$ and any finite dimensional $G$-module, when these are seen as modules for either $G_r$ or $\Gq$. The highest weight of the tilting module is such that the restriction of this module is injective for either $G_r$ or $\Gq$.

A key tool in providing these formulas comes from the use of a suitable bilinear form on the characters of finite dimensional $G$-modules. Analyzing more closely this form also allows us to give a new proof of the reciprocity between tilting modules and simple modules for $G$, which has slightly weaker assumptions than previous such proofs. Further, we also show that in a suitable reformulation, this reciprocity is equivalent to Donkin's tilting conjecture.\\

The results of this paper are formulated for a semisimple, connected, simply connected algebraic group. The reasons for working in this more restrictive setup, rather than an arbitrary connected reductive group, are the following two properties:

The group being semisimple ensures that simple modules with distinct $r$-restricted highest weights remain non-isomorphic on restriction to the $r$'th Frobenius kernel. In particular, no non-trivial $1$-dimensional $G$-module is trivial as a module for the Frobenius kernel.

The group being simply connected implies that the set of $r$-restricted weights contains a full set of representatives of the characters of the maximal torus modulo $p^r$. Consequently, any dominant weight $\lambda$ can be written as $\lambda_0 + p^r\lambda_1$ for a dominant weight $\lambda_1$ and an $r$-restricted weight $\lambda_0$. Further, this also means that any simple module for a Frobenius kernel is the restriction of a $G$-module with restricted highest weight.\\

In order to extend the results of this paper to an arbitrary connected reductive group, we can use the construction in \cite[II.1.18]{rags} to write the group as a quotient of the form $(H\times T)/Z$ where $H$ is connected and semisimple, $T$ is a torus and $Z$ is a central subgroup of the product.

The semisimple group $H$ above might not be simply connected, so to remedy this, we may need to further pass to a cover of $H$, as described in the final paragraph of \cite[II.1.17]{rags}.

\section*{Acknowledgements}

The present paper originated in my PhD dissertation, but the results have been expanded significantly in generality, in part due to many helpful suggestions from Anton Cox and Stephen Donkin who served as external examiners. Their help is gratefully acknowledged.

I would also like to thank my advisor Henning Haahr Andersen for his help throughout writing both the dissertation and this paper.

Finally, I would like to thank the anonymous referee for providing a detailed list of corrections as well as some very helpful suggestions.

\section{Notation and Preliminaries}

In this section we will introduce the notation used throughout the paper as well as some results that will be used extensively.

\subsection{List of notation}

From now on we fix the following notation. For further details on algebraic groups and their Frobenius kernels, we refer to \cite{rags}, which for the most part uses the same notation, with the main difference being that we use $\nabla(\lambda)$ and $\Delta(\lambda)$ instead of $H^0(\lambda)$ and $V(\lambda)$. For further details on finite Chevalley groups, we refer to \cite{humphreysfinlie}.

Note that throughout this paper, the term $G$-module will mean rational $G$-module.
\begin{itemize}
\item $k$ is an algebraically closed field of characteristic $p>0$.
\item $G$ is a semisimple, connected, simply connected algebraic group scheme over $k$, defined over $\mathbb{F}_p$.
\item $T\leq G$ is a maximal split torus.
\item $X = X(T)$ is the group of characters of $T$.
\item $R$ is the associated root system.
\item $S$ is a fixed basis of $R$.
\item $R^+$ is the set of positive roots corresponding to $S$.
\item $\alpha^{\vee}$ is the coroot associated to $\alpha\in R$.
\item $\langle\beta,\alpha^{\vee}\rangle$ is the natural pairing normalized such that $\langle\alpha,\alpha^{\vee}\rangle = 2$ for all $\alpha\in S$.
\item $\alpha_0$ is the highest short root of $R^+$.
\item $\rho = \frac{1}{2}\sum_{\alpha\in R^+}\alpha$ is the Weyl weight.
\item $h = \langle\rho,\alpha_0^{\vee}\rangle + 1$ is the Coxeter number of $R$.
\item $X_+ = \{\lambda\in X\mid \langle\lambda,\alpha^{\vee}\rangle \geq 0\mbox{ for all }\alpha\in R^+\}$ is the set of dominant weights.
\item $X_r = \{\lambda\in X_+\mid \langle\lambda,\alpha^{\vee}\rangle < p^r\mbox{ for all }\alpha\in S\}$ is the set of $r$-restricted weights for some integer $r\geq 1$.
\item $\Gamma_r = \{\lambda\in X_+\mid \langle\lambda,\alpha_0^{\vee}\rangle < p^r(p-h+1)\}$.
\item $\leq$ is the partial order on $X$ defined by $\lambda\leq \mu$ iff $\mu - \lambda$ is a non-negative integral linear combination of positive roots.
\item $B\leq G$ is the Borel subgroup containing $T$ corresponding to the negative roots.
\item $W$ is the Weyl group of $R$.
\item $w_0\in W$ is the longest element.
\item $\lambda^* = -w_0(\lambda)$ is the dual weight of a weight $\lambda\in X$.
\item $\nabla(\lambda) = \ind_B^G(\lambda)$ is the costandard module of highest weight $\lambda$ for $\lambda\in X_+$.
\item $L(\lambda) = \soc_G\nabla(\lambda)$ is the simple module with highest weight $\lambda\in X_+$.
\item $\Delta(\lambda) = \nabla(\lambda^*)^*$ is the Weyl (or standard) module with highest weight $\lambda\in X_+$.
\item $T(\lambda)$ is the indecomposable tilting module with highest weight $\lambda$.
\item $\St_r = \nabla((p^r-1)\rho) = L((p^r-1)\rho)\cong \Delta((p^r-1)\rho)$ is the $r$'th Steinberg module.
\item $M_{\lambda} = \{m\in M\mid t.m = \lambda(t)m\mbox{ for all }t\in T\}$ is the $\lambda$-weight space of the $G$-module $M$ for $\lambda\in X$.
\item $\mathbb{Z}[X]$ is the integral group ring of $X$ with basis $e(\lambda)$, $\lambda\in X$ such that $e(\lambda)e(\mu) = e(\lambda+\mu)$.
\item $\ZX$ is the set of $W$-fixed points of $\mathbb{Z}[X]$.
\item $[M] = \sum_{\lambda\in X}\dim(M_{\lambda})e(\lambda)\in \ZX$ is the character of the finite dimensional $G$-module $M$.
\item $F: G\to G$ is the Frobenius morphism which arises from the map $k\to k$ given by $x\mapsto x^p$.
\item $M^{(r)}$ is the $G$-module which as an additive group is the same as the $G$-module $M$, but with $G$-action composed with $F^r$.
\item $G_r = \ker F^r$ is the $r$'th Frobenius kernel.
\item $Q_r(\lambda)$ is the injective hull ($=$ projective cover) of $L(\lambda)$ as a $G_r$-module for $\lambda\in X_r$.
\item $q = p^r$ is a fixed power of $p$.
\item $\Gq = G^{F^r}$ is the fixed points of $F^r$ in $G$, which is a finite Chevalley group.
\item $P_r(\lambda)$ is the projective cover ($=$ injective hull) of $L(\lambda)$ as a $\Gq$-module for $\lambda\in X_r$.
\item $[M:L(\lambda)]_G$ is the composition multiplicity of the simple $G$-module $L(\lambda)$ in the $G$-module $M$.
\item $[M:\nabla(\lambda)]_{\nabla}$ is the multiplicity of $\nabla(\lambda)$ in a good filtration of the $G$-module $M$ (see Theorem \ref{cohcrit}).
\end{itemize}

\subsection{Preliminary results}

The results in this section will be used several times in this paper. We have included a full statement of these together with references to the original source as well as (whenever possible) a reference to either \cite{rags} or \cite{humphreysfinlie} for convenience.

The following theorem is known as Steinberg's tensor product theorem. It will be used extensively in this paper, and will therefore just be referenced as such.

\begin{thm}[{\cite[Theorem 1.1]{steinberg63}},{\cite[Proposition II.3.16]{rags}}]
Let $\lambda\in X_+$ and write $\lambda = \lambda_0 + p^r\lambda_1$ with $\lambda_0\in X_r$.

Then $L(\lambda) \cong L(\lambda_0)\otimes L(\lambda_1)^{(r)}$.
\end{thm}

Similarly to Steinberg's tensor product theorem, the following will be used several times, and will be referred to as the Andersen--Haboush tensor product theorem, due to it being discovered independently by Andersen and Haboush.

\begin{thm}[{\cite[Theorem 2.5]{andersen80}, \cite[Theorem 2.1]{haboush80}, \cite[Proposition II.3.19]{rags}}]
For each $\lambda\in X_+$ there is an isomorphism of $G$-modules $$\nabla((p^r-1)\rho + p^r\lambda)\cong \St_r \otimes \nabla(\lambda)^{(r)}.$$
\end{thm}

Recall that a good filtration of a $G$-module $M$ is a filtration with each subfactor isomorphic to $\nabla(\lambda)$ for some $\lambda\in X_+$ (see \cite[II.4]{rags}). If $M$ has a good filtration then we denote by $[M:\nabla(\lambda)]_{\nabla}$ the number of subfactors isomorphic to $\nabla(\lambda)$ in this good filtration of $M$. This is well-defined by the following.

\begin{thm}[{\cite[Corollary 1.3]{donkin81},\cite[Proposition II.4.16]{rags}}] \label{cohcrit}
Let $M$ be a $G$-module.
If $M$ admits a good filtration, then for each $\lambda\in X_{+}$, the number of factors in the filtration isomorphic to $\nabla(\lambda)$ is equal to $\dim\Hom_{G}(\Delta(\lambda),M)$.
\end{thm}

\begin{thm}[{\cite[Theorem 1]{mathieu90},\cite[Proposition II.4.19]{rags}}]\label{ja419}\label{tensorproductgoodfilt}
Let $V$ and $V'$ be $G$-modules admitting good filtrations. Then $V\otimes V'$ also admits a good filtration.
\end{thm}

The following result is well-known and will be used without further comment in the remainder of this paper.

\begin{thm}[{\cite[Proposition 10.1]{humphreyschevalley},\cite[Proposition II.10.2]{rags}}]
$\St_r$ is both projective and injective as a $G_r$-module.
\end{thm}

The following result will be referred to as the classification of tilting modules in this paper.

\begin{thm}[{\cite[Theorem 1.1]{donkin93},\cite[Proposition E.6]{rags}}]
For any $\lambda\in X_+$ there is a unique (up to isomorphism) indecomposable tilting module $T(\lambda)$ with $\dim(T(\lambda)_{\lambda}) = 1$ and such that for all $\mu\in X$ we have $T(\lambda)_{\mu}\neq 0\implies \mu\leq \lambda$.

If $Q$ is a finite dimensional tilting module, then there are uniquely determined natural numbers $n(\nu)$ such that $Q\cong\bigoplus_{\nu\in X_+}n(\nu)T(\nu)$.
\end{thm}

\begin{prop}[{\cite[Remark to E.6]{rags}}]\label{tiltingdual}
Let $\lambda\in X_+$. Then $T(\lambda)^*\cong T(\lambda^*)$ and whenever $Q$ is tilting and $L$ is simple we have $\Hom_{G}(L,Q)\cong \Hom_{G}(Q,L)$.
\end{prop}

Similarly to Steinberg's tensor product theorem, the following describes certain tilting modules as tensor products. We will refer to this as Donkin's tensor product theorem.

\begin{prop}[{\cite[Proposition 2.1]{donkin93},\cite[Lemma E.9]{rags}}]
Let $\mu = (p^r - 1)\rho + \lambda$ with $\lambda\in X_r$ and let $\nu\in X_+$. Then $T(\mu)\otimes T(\nu)^{(r)}$ is tilting, and if $T(\mu)$ is indecomposable as a $G_r$-module, then $T(\mu)\otimes T(\nu)^{(r)}\cong T(\mu + p^r\nu)$.
\end{prop}

The condition that $T(\mu)$ is indecomposable as a $G_r$-module when $\mu$ is as in the above proposition is in fact conjectured to always hold (see \cite[Conjecture 2.2]{donkin93}) and is known as Donkin's tilting conjecture. It is known to hold if $p\geq 2h-2$.

\begin{rem}
With a few exceptions, whenever a result in this paper includes the condition $p\geq 2h-2$ this can be removed if one assumes Donkin's tilting conjecture. We will make sure to note whenever this is not the case.
\end{rem}

\begin{thm}[{\cite[E.9]{rags}}]\label{tiltingsocle}
If $\lambda\in X_r$ and $T((p^r-1)\rho + \lambda)$ is indecomposable as a $G_r$-module, then $$\soc_{G}(T((p^r-1)\rho + \lambda)) = \soc_{G_r}(T((p^r-1)\rho + \lambda)) = L((p^r-1)\rho - \lambda^*).$$ In particular, this holds for all $\lambda\in X_r$ if $p\geq 2h-2$.
\end{thm}

\begin{prop}[{\cite[Proposition 2.4]{donkin93},\cite[Lemma E.8]{rags}}]\label{tiltinginj}
If $\lambda\in X_+$ then $T(\lambda)$ is injective as a $G_r$-module if and only if $\langle\lambda,\alpha^{\vee}\rangle \geq p^r - 1$ for all $\alpha\in S$.
\end{prop}

Similarly to the above, we have the following for $\Gq$-modules. It follows for example from the above by applying \cite[Theorem 2.3]{drupieski13}.

\begin{prop}\label{tiltingproj}
$T((p^r-1)\rho + \lambda)$ is injective and projective as a $\Gq$-module for any $\lambda\in X_+$.
\end{prop}

The following two results give conditions on a $G$-module $M$ which guarantee that $\St_r\otimes M$ has a good filtration.

\begin{thm}[{\cite[Theorem 4.3.2]{kildetoftnakano15}}]\label{lambdaalpha0}
Assume that $p\geq h$ and let $\lambda\in \Gamma_r$.
Then $\St_r\otimes L(\lambda)$ has a good filtration.
\end{thm}

\begin{prop}[{\cite[Proposition 4.4]{andersen01}}]\label{tensorproductsteinberggood}
Let $M$ and $N$ be $G$-modules such that both $\St_r\otimes M$ and $\St_r\otimes N$ have good filtrations. Then $\St_r\otimes (M\otimes N)$ also has a good filtration.
\end{prop}

\section{Good filtrations on tensor products}

In this section, we gather various results showing that certain tensor products have good filtrations.

\subsection{Tensoring with tilting modules}

\begin{thm}\label{goodfilttensorwithtilting}
Let $M$ be a $G$-module such that $\St_r\otimes M$ has a good filtration and let $\lambda\in X_+$. Then $T((p^r-1)\rho + \lambda)\otimes M$ has a good filtration.
\end{thm}

\begin{proof}
Since $\St_r\otimes T(\lambda)$ is tilting by Theorem \ref{tensorproductgoodfilt} we see that it has $T((p^r-1)\rho + \lambda)$ as a direct summand by the classification of tilting modules. Thus $T((p^r-1)\rho + \lambda)\otimes M$ is a direct summand of $\St_r\otimes T(\lambda)\otimes M$ which has a good filtration by Proposition \ref{tensorproductsteinberggood}. Since direct summands of modules with good filtrations themselves have good filtrations, this finishes the proof.
\end{proof}

A useful consequence of the above is the following.

\begin{cor}\label{exttensor}
Let $\lambda,\mu,\nu\in X_+$ and let $M$ be a $G$-module such that $\St_r\otimes M$ has a good filtration. Then $$\dim\Hom_G(T((p^r-1)\rho + \lambda),M\otimes\nabla(\nu)\otimes\nabla(\mu)^{(r)}) = [T((p^r-1)\rho + \lambda^*)\otimes \nabla(\mu)^{(r)}\otimes M:\nabla(\nu^*)]_{\nabla}$$ and $$\Ext_G^i(T((p^r-1)\rho + \lambda),M\otimes\nabla(\nu)\otimes\nabla(\mu)^{(r)}) = 0$$ for all $i\geq 1$.
\end{cor}

\begin{proof}
Rewrite $$\Ext_G^i(T((p^r-1)\rho + \lambda),M\otimes\nabla(\nu)\otimes\nabla(\mu)^{(r)}) \cong \Ext_G^i(\Delta(\nu^*),T((p^r-1)\rho + \lambda^*)\otimes\nabla(\mu)^{(r)}\otimes M)$$ and note that by Proposition \ref{tensorproductsteinberggood} and the Andersen--Haboush tensor product theorem $\St_r \otimes \nabla(\mu)^{(r)}\otimes M$ has a good filtration, so the claim now follows from Theorem \ref{goodfilttensorwithtilting} together with Theorem \ref{cohcrit}.
\end{proof}

\subsection{Weights in $\Gamma_r$}

Recall that we define $\Gamma_r = \{\lambda\in X_+\mid\langle\lambda,\alpha_0^{\vee}\rangle < p^r(p-h+1)\}$ and note that if $\lambda\in \Gamma_r$ and $\mu\leq\lambda$ for some $\mu\in X_+$ then $\mu\in \Gamma_r$. Similarly, $\lambda\in \Gamma_r\Leftrightarrow \lambda^*\in\Gamma_r$.

Note also that if $p < h$ then $\Gamma_r = \emptyset$ and if $p\geq 2h-2$ then $X_r\subseteq \Gamma_r$.

We can now reformulate \cite[Corollary II.5.6]{rags} to the statement that whenever $\lambda\in \Gamma_0$ then $\nabla(\lambda)$ is simple.

\begin{lm}\label{gammar}
If $\lambda\in \Gamma_r$ and $\lambda = \lambda_0 + p^u\lambda_1$ with $\lambda_0\in X_u$ for some $u\leq r$ then $\lambda_1\in \Gamma_{r-u}$.

In particular, if $u = r$ then $\lambda_1\in \Gamma_0$ and $\nabla(\lambda_1)$ is simple.
\end{lm}

\begin{proof}
This follows straight from the definition (the last claim follows from the above observation).
\end{proof}

\begin{prop}\label{tensoristilting}
If $\lambda\in \Gamma_r$ and $\mu\in X_+$ then $T((p^r-1)\rho + \mu)\otimes \nabla(\lambda)$ and $T((p^r-1)\rho + \mu)\otimes \Delta(\lambda)$ are tilting.
\end{prop}

\begin{proof}
By Theorem \ref{goodfilttensorwithtilting} it suffices to show that $\St_r\otimes \Delta(\lambda)$ has a good filtration. But since $\lambda\in\Gamma_r$, whenever $L(\nu)$ is a composition factor of $\Delta(\lambda)$ we also have $\nu\in\Gamma_r$ and thus $\St_r\otimes L(\nu)$ has a good filtration by Theorem \ref{lambdaalpha0} which finishes the proof.
\end{proof}

\begin{cor}\label{tensorisomorphic}
If $\lambda\in\Gamma_r$ and $\mu\in X_+$ then there is an isomorphism $$T((p^r-1)\rho + \mu)\otimes\nabla(\lambda)\cong T((p^r-1)\rho + \mu)\otimes \Delta(\lambda).$$
\end{cor}

\begin{proof}
By Proposition \ref{tensoristilting} both $T((p^r-1)\rho + \mu)\otimes\nabla(\lambda)$ and $T((p^r-1)\rho + \mu)\otimes \Delta(\lambda)$ are tilting, and they have the same characters since $\nabla(\lambda)$ and $\Delta(\lambda)$ do. Hence they are isomorphic by the classification of tilting modules.
\end{proof}

\section{Tilting- and socle-numbers}

The main goal of this section is to study the decomposition of modules of the form $\St_r\otimes\nabla(\lambda)$, but for purposes of later induction arguments, we will from the beginning study a more general class of modules, namely those of the form $T((p^r-1)\rho + \mu)\otimes\nabla(\lambda)$.

By Proposition \ref{tensoristilting} and Theorem \ref{tensorproductgoodfilt}, whenever $\mu\in X_+$ and $\lambda\in \Gamma_r$ we can write $$T((p^r-1)\rho + \mu)\otimes\nabla(\lambda) = \bigoplus_{\nu\in X_+}t_{\lambda,\mu}^r(\nu)T(\nu)$$
for suitable natural numbers $t_{\lambda,\mu}^r(\nu)$. We will be referring to these numbers as the tilting-numbers.

In particular, $\St_r\otimes \nabla(\lambda)$ is tilting whenever $\lambda\in\Gamma_r$ so we can write $$\St_r\otimes\nabla(\lambda) = \bigoplus_{\nu\in X_+}t_{\lambda}^r(\nu)T(\nu)$$ where $t_{\lambda}^r(\nu) = t_{\lambda,0}^r(\nu)$.

As we will see, the above decomposition is in fact (mostly) determined by the socle of the module, so for $\lambda,\mu\in X_+$ we define natural numbers $s_{\lambda,\mu}(\nu)$ by $$\soc_G(T((p^r-1)\rho + \mu)\otimes\nabla(\lambda)) = \bigoplus_{\nu\in X_+}s_{\lambda,\mu}^r(\nu)L(\nu).$$
We will be referring to these numbers as the socle-numbers, and similarly to above we set $s_{\lambda}^r(\nu) = s_{\lambda,0}^r(\nu)$.

Note that if $G$ is either $SL_2$ or $SL_3$ then the decomposition of the module $\St_r\otimes\nabla(\lambda)$ for $\lambda\in \Gamma_r$ can be computed using the results of \cite{dotyhenke05}, \cite{bdm11} and \cite{bdm15} by using that $\St_r\otimes L(\nu)$ is tilting for each composition factor $L(\nu)$ of $\nabla(\lambda)$ and that the standard character data for these groups is known.

\subsection{Comparing tilting- and socle-numbers}

For any natural number $r$ define a map $w_r: X_+\to X_+$ by $$w_r(\lambda) = (p^r-1)\rho - \lambda_0^* + p^r\lambda_1$$ where $\lambda = \lambda_0 + p^r\lambda_1$ with $\lambda_0\in X_r$. Note that $w_r$ is not linear, but we do have the following.

\begin{lm}\label{wrproperties}
\begin{enumerate}
\item $w_r^2 = \id$.
\item $w_r(X_r) = X_r$.
\item If $\lambda = \lambda_0 + p^{u}\lambda_1$ with $\lambda_0\in X_{u}$ and $r > u$ then $$w_r(\lambda) = w_{u}(\lambda_0) + p^{u}w_{r-u}(\lambda_1).$$
\item If $\lambda = \lambda_0 + p\lambda_1 + \cdots + p^{r-1}\lambda_{r-1}$ with all $\lambda_i\in X_1$ then $$w_r(\lambda) = w_1(\lambda_0) + pw_1(\lambda_1) + \cdots + p^{r-1}w_1(\lambda_{r-1}).$$
\end{enumerate}
\end{lm}

\begin{proof}
$1.$ and $2.$ are clear from the definition.

For $3.$ write $\lambda_1 = \lambda_0' + p^{r-u}\lambda_1'$ with $\lambda_0'\in X_{r-u}$ and note that then $\lambda = (\lambda_0 + p^{u}\lambda_0') + p^r\lambda_1'$ with $\lambda_0+p^{u}\lambda_0'\in X_r$ so we get $$w_r(\lambda) = (p^r-1)\rho - (\lambda_0 + p^{u}\lambda_0')^* + p^r\lambda_1'$$ $$= (p^{u}-1)\rho + p^{u}(p^{r-u}-1)\rho - \lambda_0^* - p^{u}\lambda_0'^* + p^{r-u}p^{u}\lambda_1'$$ $$= w_{u}(\lambda_0) + p^{u}w_{r-u}(\lambda_1)$$ as claimed.

$4.$ follows directly from $3.$ by induction.
\end{proof}

We would like to be able to compare the $s_{\lambda,\mu}^r$ and $t_{\lambda,\mu}^r$ (which we can regard as functions $X_+\to \Z_+$). For that purpose we need the following which slightly extends Theorem \ref{tiltingsocle}.

\begin{lm}\label{tiltingsoclelarger}
Assume that $p\geq 2h-2$ and let $\lambda \in \Gamma_r$. Then $\soc_{G}(T((p^r-1)\rho + \lambda)) = L(w_r(\lambda))$.
\end{lm}

\begin{proof}
Write $\lambda = \lambda_0 + p^r\lambda_1$ with $\lambda_0\in X_r$. Then by Donkin's tensor product theorem we have $$T((p^r-1)\rho + \lambda)\cong T((p^r-1)\rho + \lambda_0)\otimes T(\lambda_1)^{(r)}.$$

Now we get by Theorem \ref{tiltingsocle} that $\soc_{G_r}(T((p^r-1)\rho + \lambda_0)) = L((p^r-1)\rho - \lambda_0^*)$, so for $\mu = \mu_0 + p^r\mu_1\in X_+$ with $\mu_0\in X_r$ we get, by Steinberg's tensor product theorem, $$\Hom_{G}(L(\mu),T((p^r-1)\rho + \lambda))$$ $$\cong \Hom_{G/G_r}(L(\mu_1)^{(r)},\Hom_{G_r}(L(\mu_0),T((p^r-1)\rho + \lambda_0))\otimes T(\lambda_1)^{(r)})$$ $$\cong \begin{cases}\Hom_{G}(L(\mu_1),T(\lambda_1)) & \mbox{if } \mu_0 = (p^r-1)\rho - \lambda_0^* \\ 0 & \mbox{else} \end{cases}$$ so we just need to show that $\soc_{G}(T(\lambda_1)) = L(\lambda_1)$. In fact, we claim that $T(\lambda_1)$ is simple, which clearly suffices.

To see this, we note that by Lemma \ref{gammar} since $\lambda\in \Gamma_r$, $\nabla(\lambda_1)$ is simple, and hence also that $T(\lambda_1)$ is simple, being the unique indecomposable tilting module of highest weight $\lambda_1$.

Thus we have shown that $\soc_{G}(T((p^r-1)\rho + \lambda)) = L((p^r-1)\rho - \lambda_0^* + p^r\lambda_1) = L(w_r(\lambda))$ as was the claim.
\end{proof}

The way to apply the above lemma is given in the following.

\begin{prop}\label{trnonzero}
If $\lambda,\mu,\nu\in X_+$ with $\lambda+\mu\in\Gamma_r$ and $t_{\lambda,\mu}^r(\nu)\neq 0$ then $\nu = (p^r-1)\rho + \nu'$ with $\nu'\in \Gamma_r$.
\end{prop}

\begin{proof}
Since $T((p^r-1)\rho + \mu)\otimes \nabla(\lambda)$ is injective as a $G_r$-module by Theorem \ref{tiltinginj}, the same must be true for $T(\nu)$ since this is a direct summand of $T((p^r-1)\rho + \mu)\otimes \nabla(\lambda)$ by assumption. Now it follows by Proposition \ref{tiltinginj} that $\nu = (p^r-1)\rho +\nu'$ for some $\nu'\in X_+$.

Since $\nu$ must be a weight of $T((p^r-1)\rho + \mu)\otimes \nabla(\lambda)$ we must have $\nu \leq (p^r-1)\rho + \mu + \lambda$, and hence we get $\nu' \leq \lambda + \mu$, which implies the claim.
\end{proof}

Combining the above results, we get the following.

\begin{cor}\label{relationts}
Assume that $p\geq 2h-2$. Then for any $\lambda,\mu,\nu\in X_+$ with $\lambda + \mu\in \Gamma_r$ we have $s_{\lambda,\mu}^r(\nu) = t_{\lambda,\mu}^r((p^r-1)\rho + w_r(\nu))$.
\end{cor}

\begin{proof}
This follows directly from combining Lemma \ref{tiltingsoclelarger} and Proposition \ref{trnonzero} (recall that we have $w_r(w_r(\nu)) = \nu$ by Lemma \ref{wrproperties}).
\end{proof}

We also get the following weaker result when we remove the assumption that $p\geq 2h-2$.

\begin{cor}
Let $\lambda,\mu,\nu\in X_+$ with $\lambda+\mu\in \Gamma_r$. Then $s_{\lambda,\mu}^r(\nu) \geq t_{\lambda,\mu}^r((p^r-1)\rho + w_r(\nu))$.
\end{cor}

\begin{proof}
Since $L(\nu)\subseteq \soc_GT((p^r-1)\rho + w_r(\nu))$, this follows similarly to the above.
\end{proof}

The following gives a condition under which the conclusion of Proposition \ref{trnonzero} can be strengthened. This will be useful for reducing certain calculations to the $r=1$ case later.

\begin{prop}\label{trnonzerorestricted}
If $\lambda,\mu,\nu\in X_+$ with $\langle\lambda+\mu,\alpha_0^{\vee}\rangle < p^r$ and $t_{\lambda,\mu}^r(\nu)\neq 0$ then $\nu = (p^r-1)\rho + \nu'$ with $\nu'\in X_r$.
\end{prop}

\begin{proof}
By the same arguments as in the proof of Proposition \ref{trnonzero} we see that $\nu = ((p^r-1)\rho + \nu'$ with $\nu'\leq \lambda + \mu$. In particular, if $\nu' = \nu_0 + p^r\nu_1$ with $\nu_0\in X_r$ then $\langle\nu_1,\alpha_0^{\vee}\rangle \leq \frac{\langle\lambda+\mu,\alpha_0^{\vee}\rangle}{p^r} < 1$ so we must have $\nu_1 = 0$ and thus $\nu'\in X_r$ as claimed.
\end{proof}

We have now seen that instead of computing the $t_{\lambda,\mu}$ we can compute the $s_{\lambda,\mu}$, at least when $p\geq 2h-2$. The following will prove very useful for this purpose.

\begin{thm}\label{soclegoodfilt}
Assume that $p\geq 2h - 2$ and let $\mu,\lambda,\nu\in X_+$ with $\lambda+\mu\in\Gamma_r$ and $s_{\lambda,\mu}^r(\nu)\neq 0$. Then $T((p^r-1)\rho + \gamma)\otimes L(\nu)$ and $T((p^r-1)\rho + \gamma)\otimes L(\nu^*)$ are tilting for any $\gamma\in X_+$.
\end{thm}

\begin{proof}
By Theorem \ref{goodfilttensorwithtilting} it suffices to show that $\St_r\otimes L(\nu)$ and $\St_r\otimes L(\nu^*)$ are tilting.

By Corollary \ref{relationts} and Proposition \ref{trnonzero} we see that $\nu = w_r(\sigma)$ for some dominant weight $\sigma$ with $\sigma\in\Gamma_r$. So if we write $\sigma = \sigma_0 + p^r\sigma_1$ with $\sigma_0\in X_r$ we have $\nu = \nu_0 + p^r\nu_1$ with $\nu_0 = (p^r-1)\rho - \sigma_0^*\in X_r$ and $\nu_1 = \sigma_1$. In particular, we have $\nu_1\in\Gamma_0$ and hence $L(\nu_1)\cong\nabla(\nu_1)\cong T(\nu_1)$ by Lemma \ref{gammar}.

By Steinberg's tensor product theorem we now have $$L(\nu)\cong L(\nu_0)\otimes L(\nu_1)^{(r)}\cong L(\nu_0)\otimes T(\nu_1)^{(r)}$$ so the claims follow by combining Theorem \ref{lambdaalpha0}, Proposition \ref{tensorproductsteinberggood} and the Andersen--Haboush tensor product theorem.
\end{proof}

\subsection{Duality}

\begin{prop}
If $\lambda\in \Gamma_r$ and and $\mu,\nu\in X_+$ then $s_{\lambda,\mu}^r(\nu) = s_{\lambda^*,\mu^*}^r(\nu^*)$ and $t_{\lambda,\mu}^r(\nu) = t_{\lambda^*,\mu^*}^r(\nu^*)$.
\end{prop}

\begin{proof}
We apply Proposition \ref{tiltingdual} and Proposition \ref{tensoristilting} to get $$s_{\lambda,\mu}^r(\nu) = \dim\Hom_G(L(\nu),T((p^r-1)\rho + \mu)\otimes\nabla(\lambda)) = \dim\Hom_G(T((p^r-1)\rho + \mu)\otimes\nabla(\lambda),L(\nu))$$ $$= \dim\Hom_G(L(\nu^*),T((p^r-1)\rho + \mu^*)\otimes\Delta(\lambda^*))$$ and then apply Corollary \ref{tensorisomorphic} to see that this equals $$\dim\Hom_G(L(\nu^*),T((p^r-1)\rho + \mu^*)\otimes\nabla(\lambda^*)) = s_{\lambda^*,\mu^*}^r(\nu^*).$$

The second claim also follows from Corollary \ref{tensorisomorphic} since any $T(\nu)$ will occur in some tilting module $M$ the same number of times as $T(\nu)^*\cong T(\nu^*)$ will occur in $M^*$.
\end{proof}

\subsection{Computing socle-numbers}

In order to determine necessary conditions on $\nu\in X_+$ to have $s_{\lambda,\mu}^r(\nu)\neq 0$ for some given $\lambda,\mu\in X_+$ we can make use of the following.

Note that while this only provides an inequality, it requires a lot fewer assumptions that many of the similar results we will obtain here, and for the $\mu = 0$ case, all the characters involved are already known.

\begin{prop}\label{necessary}
Let $\lambda,\mu,\nu\in X_+$. Then we have
$$s_{\lambda,\mu}^r(\nu)\leq [T((p^r-1)\rho + \mu)\otimes \nabla(\nu^*):\nabla(\lambda^*)]_{\nabla}.$$ In particular, $$s_{\lambda}^r(\nu) \leq [\St_r\otimes\nabla(\nu^*):\nabla(\lambda^*)]_{\nabla}.$$
\end{prop}

\begin{proof}
We have $$s_{\lambda,\mu}^r(\nu) = \dim\Hom_G(L(\nu),T((p^r-1)\rho + \mu)\otimes\nabla(\lambda)) = \dim\Hom_{G}(\Delta(\lambda^*),T((p^r-1)\rho + \mu)\otimes L(\nu^*))$$ and the inclusion $L(\nu^*)\hookrightarrow\nabla(\nu^*)$ gives the inequality $$\dim\Hom_{G}(\Delta(\lambda^*),T((p^r-1)\rho + \mu)\otimes L(\nu^*))\leq \dim\Hom_{G}(\Delta(\lambda^*),T((p^r-1)\rho + \mu)\otimes\nabla(\nu^*))$$ and by Theorem \ref{cohcrit} we have $$\dim\Hom_{G}(\Delta(\lambda^*),T((p^r-1)\rho + \mu)\otimes\nabla(\nu^*)) = [T((p^r-1)\rho + \mu)\otimes\nabla(\nu^*):\nabla(\lambda^*)]_{\nabla}$$ since $T((p^r-1)\rho + \mu)\otimes \nabla(\nu^*)$ has a good filtration by Theorem \ref{tensorproductgoodfilt}. This proves the claim.
\end{proof}

The following reduces the computation of socle-numbers to standard character data for $G$ under suitable conditions.

\begin{thm}\label{necessarybigp}
Assume that $p\geq 2h-2$ and let $\nu = \nu_0 + p^r\nu_1\in X_+$ with $\nu_0\in X_r$ such that $s_{\lambda,\mu}^r(\nu)\neq 0$ for some $\lambda,\mu\in X_+$ with $\lambda+\mu\in\Gamma_r$. Then we have
\begin{enumerate}
\item $s_{\lambda,\mu}^r(\nu) = [T((p^r-1)\rho + \mu)\otimes L(\nu^*):\nabla(\lambda^*)]_{\nabla}$.
\item If further $\mu\in X_r$ then $s_{\lambda,\mu}^r(\nu) = [T((p^r-1)\rho + \mu + p^r\nu_1^*)\otimes L(\nu_0^*):\nabla(\lambda^*)]_{\nabla}$.
\item $s_{\lambda}^r(\nu) = [\nabla((p^r-1)\rho + p^r\nu_1^*)\otimes L(\nu_0^*):\nabla(\lambda^*)]_{\nabla}$.
\end{enumerate}
\end{thm}

\begin{proof}
If $p\geq 2h-2$ and $s_{\lambda,\mu}^r(\nu)\neq 0$ then $T((p^r-1)\rho + \mu)\otimes L(\nu^*)$ has a good filtration by Theorem \ref{soclegoodfilt}, and we have $$[T((p^r-1)\rho + \mu)\otimes L(\nu^*):\nabla(\lambda^*)]_{\nabla} = \dim\Hom_{G}(\Delta(\lambda^*),T((p^r-1)\rho + \mu)\otimes L(\nu^*))$$ $$= \dim\Hom_G(L(\nu),T((p^r-1)\rho + \mu)\otimes\nabla(\lambda))$$ by Theorem \ref{cohcrit}, so the first claim follows.

For the second claim, we use that by Steinberg's tensor product theorem we have $L(\nu^*)\cong L(\nu_0^*)\otimes L(\nu_1^*)^{(r)}$, and from Proposition \ref{trnonzero} together with Corollary \ref{relationts} we see that $\nu_1^*\in\Gamma_0$ and hence we get $L(\nu_1^*)\cong \nabla(\nu_1^*)$ by Lemma \ref{gammar} which also implies that $L(\nu_1^*)\cong T(\nu_1^*)$. Applying Donkin's tensor product theorem we thus have $T((p^r-1)\rho + \mu)\otimes L(\nu_1^*)^{(r)}\cong T((p^r-1)\rho + \mu + p^r\nu_1^*)$, so $T((p^r-1)\rho +\mu)\otimes L(\nu^*)\cong T((p^r-1)\rho + \mu + p^r\nu_1^*)\otimes L(\nu_0^*)$ which gives the claim.

The final claim follows similarly to above by applying the Andersen--Haboush tensor product theorem.
\end{proof}

\begin{cor}\label{srgivencharacter}
Assume that $p\geq 2h-2$, $\lambda,\mu\in X_+$ with $\lambda+\mu\in \Gamma_r$ and let $\nu\in X_+$ with $s_{\lambda,\mu}^r(\nu)\neq 0$. Write $$[\nabla(\nu^*)] = [L(\nu^*)] +  \sum_{\psi\in X_+,\, \psi\neq\nu^*}a_{\psi}[L(\psi)].$$

Then $$s_{\lambda,\mu}^r(\nu) = [T((p^r-1)\rho + \mu)\otimes \nabla(\nu^*):\nabla(\lambda^*)]_{\nabla} - \sum_{\psi\in X_+}a_{\psi}s_{\lambda,\mu}^r(\psi^*).$$ In particular, $$s_{\lambda}^r(\nu) = [\St_r\otimes \nabla(\nu^*):\nabla(\lambda^*)]_{\nabla} - \sum_{\psi\in X_+}a_{\psi}s_{\lambda}^r(\psi^*),$$ and if $a_{\psi}\neq 0$ implies that $s_{\lambda,\mu}^r(\psi^*)=0$ then $s_{\lambda,\mu}^r(\nu) = [T((p^r-1)\rho + \mu)\otimes\nabla(\nu^*):\nabla(\lambda^*)]_{\nabla}$.
\end{cor}

\begin{proof}
First, we claim that $\St_r\otimes L(\psi)$ has a good filtration whenever $a_{\psi}\neq 0$.

To see this, note that for such $\psi$ we have $\psi\leq \nu^*$, and then the claim is clear from the proof of Theorem \ref{soclegoodfilt}.

By Theorem \ref{necessarybigp} we see that $s_{\lambda,\mu}^r(\nu) = [T((p^r-1)\rho + \mu)\otimes L(\nu^*):\nabla(\lambda^*)]_{\nabla}$, and this is completely determined by the character of $T((p^r-1)\rho + \mu)\otimes L(\nu^*)$. Thus we have $$s_{\lambda,\mu}^r(\nu) = [T((p^r-1)\rho + \mu)\otimes\nabla(\nu^*):\nabla(\lambda^*)]_{\nabla} - \sum_{\psi\in X_+}a_{\psi}[T((p^r-1)\rho + \mu)\otimes L(\psi):\nabla(\lambda^*)]_{\nabla}.$$

Now we apply Theorem \ref{necessarybigp} again to see that whenever $a_{\psi}\neq 0$ we have $s_{\lambda,\mu}^r(\psi^*) = [T((p^r-1)\rho + \mu)\otimes L(\psi):\nabla(\lambda^*)]_{\nabla}$, which finishes the proof.
\end{proof}

In the above, most of the assumptions were there to ensure that $T((p^r-1)\rho + \mu)\otimes L(\nu^*)$ had a good filtration. We can therefore also formulate the following which has somewhat different assumptions.

\begin{thm}
Let $\lambda,\mu\in X_+$ and $\nu\in \Gamma_r$. Then $$s_{\lambda,\mu}^r(\nu) = [T((p^r-1)\rho + \mu)\otimes L(\nu^*):\nabla(\lambda^*)]_{\nabla}.$$ In particular, if $\nu = \nu_0 + p^r\nu_1$ with $\nu_0\in X_r$ then $$s_{\lambda}^r(\nu) = [\nabla((p^r-1)\rho + p^r\nu_1^*)\otimes L(\nu_0^*):\nabla(\lambda^*)]_{\nabla}.$$
\end{thm}

\begin{proof}
By Theorem \ref{lambdaalpha0} and Theorem \ref{goodfilttensorwithtilting} $T((p^r-1)\rho + \mu)\otimes L(\nu^*)$ has a good filtration, so by Theorem \ref{cohcrit} we get $$s_{\lambda,\mu}^r(\nu) = \dim\Hom_G(L(\nu),T((p^r-1)\rho + \mu)\otimes\nabla(\lambda))$$ $$= \dim\Hom_G(\Delta(\lambda^*),T((p^r-1)\rho + \mu)\otimes L(\nu^*)) = [T((p^r-1)\rho + \mu)\otimes L(\nu^*):\nabla(\lambda^*)]_{\nabla}$$ as claimed.

The second claim follows by applying the Andersen--Haboush tensor product theorem.
\end{proof}

And we also get a similar corollary.

\begin{cor}
Let $\lambda,\mu\in X_+$ and $\nu\in \Gamma_r$. Write $$[\nabla(\nu^*)] = [L(\nu^*)] +  \sum_{\psi\in X_+,\, \psi\neq\nu^*}a_{\psi}[L(\psi)].$$ Then $$s_{\lambda,\mu}^r(\nu) = [T((p^r-1)\rho + \mu)\otimes \nabla(\nu^*):\nabla(\lambda^*)]_{\nabla} - \sum_{\psi\in X_+}a_{\psi}s_{\lambda,\mu}^r(\psi^*).$$ In particular, $$s_{\lambda}^r(\nu) = [\St_r\otimes \nabla(\nu^*):\nabla(\lambda^*)]_{\nabla} - \sum_{\psi\in X_+}a_{\psi}s_{\lambda}^r(\psi^*),$$ and if $a_{\psi}\neq 0$ implies that $s_{\lambda,\mu}^r(\psi^*) = 0$ then $s_{\lambda,\mu}^r(\nu) = [T((p^r-1)\rho + \mu)\otimes \nabla(\nu^*):\nabla(\lambda^*)]_{\nabla}$.
\end{cor}

\begin{proof}
Since $T((p^r-1)\rho + \mu)\otimes L(\nu^*)$ has a good filtration by Theorem \ref{lambdaalpha0} and Theorem \ref{goodfilttensorwithtilting} this follows in the same way as Corollary \ref{srgivencharacter}.
\end{proof}

\subsection{Inductive formulas}

In this section, we will give ways to relate $t_{\lambda}^r$ to $t_{\lambda'}^u$ for suitable $\lambda'$ and $u\leq r$, in some cases reducing everything to the $r=1$ case.

Note that the condition that $\nabla(\lambda)\cong\nabla(\lambda_0)\otimes\nabla(\lambda_1)^{(u)}$ in the following is a very strong assumption, but that it at least holds if either $\lambda_0 = (p^u-1)\rho$ or if all of $\nabla(\lambda)$, $\nabla(\lambda_0)$ and $\nabla(\lambda_1)$ are simple (the first follows by the Andersen--Haboush tensor product theorem and the second by Steinberg's tensor product theorem).

\begin{prop}\label{inductively}
Assume that $p\geq 2h-2$, let $\lambda\in\Gamma_r$ and write $\lambda = \lambda_0 + p^u\lambda_1$ with $\lambda_0\in X_u$ for some $u\leq r$. Assume that $\nabla(\lambda)\cong \nabla(\lambda_0)\otimes\nabla(\lambda_1)^{(u)}$. Then for any $\nu = \nu_0 + p^u\nu_1\in X_+$ with $\nu_0\in (p^u-1)\rho + X_u$ we have $$t_{\lambda}^r(\nu) = \sum_{\sigma\in \Gamma_0}\sum_{\mu \in \Gamma_{r-u}}t_{\lambda_0}^u(\nu_0 + p^u\sigma)t_{\lambda_1}^{r-u}((p^{r-u}-1)\rho + \mu)t_{\sigma,\mu}^{r-u}(\nu_1).$$
\end{prop}

\begin{proof}
Since $\St_r\cong \St_u\otimes\St_{r-u}^{(u)}$ by Steinberg's tensor product theorem we can write 
$$\St_r\otimes\nabla(\lambda) \cong (\St_u\otimes \nabla(\lambda_0))\otimes (\St_{r-u}\otimes\nabla(\lambda_1))^{(u)}$$ $$\cong \left(\bigoplus_{\gamma\in X_+}t_{\lambda_0}^u(\gamma)T(\gamma)\right)\otimes\left(\bigoplus_{\psi\in X_+}t_{\lambda_1}^{r-u}(\psi)T(\psi)\right)^{(u)}$$
and for each $\gamma$ with $t_{\lambda_0}^u(\gamma)\neq 0$ we can write $\gamma = (p^u-1)\rho + \gamma_0' + p^u\sigma$ with $\gamma_0'\in X_u$ and $\sigma\in\Gamma_0$ by Proposition \ref{trnonzero} and Lemma \ref{gammar}. Set $\gamma_0 = (p^u-1)\rho + \gamma_0'$ so $T(\gamma)\cong T(\gamma_0)\otimes T(\sigma)^{(u)}\cong T(\gamma_0)\otimes \nabla(\sigma)^{(u)}$ by Donkin's tensor product theorem (since $\sigma\in\Gamma_0$ so $T(\sigma)\cong L(\sigma)\cong \nabla(\sigma)$). We can then rearrange the above to get 
$$\bigoplus_{\gamma_0\in (p^u-1)\rho + X_u}\bigoplus_{\sigma\in\Gamma_0}\bigoplus_{\psi\in X_+}t_{\lambda_0}^u(\gamma_0+p^u\sigma)t_{\lambda_1}^{r-u}(\psi)\left(T(\gamma_0)\otimes\left(T(\psi)\otimes \nabla(\sigma)\right)^{(u)}\right)$$
and similarly to above, whenever $t_{\lambda_1}^{r-u}(\psi)\neq 0$ we can write $\psi = (p^{r-u}-1)\rho + \mu$ with $\mu\in\Gamma_{r-u}$ so we can rewrite further by expanding the tensor product $T(\psi)\otimes\nabla(\sigma) = T((p^{r-u}-1)\rho + \mu)\otimes\nabla(\sigma)$
$$\bigoplus_{\gamma_0\in (p^u-1)\rho + X_u}\bigoplus_{\sigma\in\Gamma_0}\bigoplus_{\mu\in\Gamma_{r-u}}\bigoplus_{\varphi\in X_+}t_{\lambda_0}^u(\gamma_0 + p^u\sigma)t_{\lambda_1}^{r-u}((p^{r-u}-1)\rho + \mu)t_{\sigma,\mu}^{r-u}(\varphi)\left(T(\gamma_0)\otimes T(\varphi)^{(u)}\right)$$
$$\cong\bigoplus_{\gamma_0\in (p^u-1)\rho + X_u}\bigoplus_{\sigma\in\Gamma_0}\bigoplus_{\mu\in\Gamma_{r-u}}\bigoplus_{\varphi\in X_+}t_{\lambda_0}^u(\gamma_0 + p^u\sigma)t_{\lambda_1}^{r-u}((p^{r-u}-1)\rho + \mu)t_{\sigma,\mu}^{r-u}(\varphi)T(\gamma_0 + p^u\varphi)$$ 
from which the claim follows by fixing $\gamma_0 = \nu_0$ and $\varphi = \nu_1$.
\end{proof}

\begin{cor}\label{inductivelycor}
Assume that $p\geq 2h-2$, let $\lambda\in\Gamma_r$ and write $\lambda = \lambda_0 + p^u\lambda_1$ with $\lambda_0\in X_u$ for some $u\leq r$. Assume that $\nabla(\lambda)\cong \nabla(\lambda_0)\otimes\nabla(\lambda_1)^{(u)}$ and further that $t_{\lambda_0}^u(\nu)\neq 0\implies \nu\in (p^u-1)\rho + X_u$. Then for any $\nu\in X_+$ with $t_{\lambda}^r(\nu)\neq 0$ we can write $\nu = \nu_0 + p^u\nu_1$ with $\nu_0\in (p^u-1)\rho + X_u$ and $\nu_1\in (p^{r-u}-1)\rho + X_+$ and for such $\nu$ we have $t_{\lambda}^r(\nu) = t_{\lambda_0}^u(\nu_0)t_{\lambda_1}^{r-u}(\nu_1)$.
\end{cor}

\begin{proof}
Since $\St_r\cong \St_u\otimes\St_{r-u}^{(u)}$ by Steinberg's tensor product theorem we can write 
$$\St_r\otimes\nabla(\lambda) \cong (\St_u\otimes \nabla(\lambda_0))\otimes (\St_{r-u}\otimes\nabla(\lambda_1))^{(u)}$$ $$\cong \left(\bigoplus_{\gamma\in X_+}t_{\lambda_0}^u(\gamma)T(\gamma)\right)\otimes\left(\bigoplus_{\psi\in X_+}t_{\lambda_1}^{r-u}(\psi)T(\psi)\right)^{(u)}$$ and since by assumption each $\gamma$ with $t_{\lambda_0}^u(\gamma)\neq 0$ can be written as $\gamma = (p^u-1)\rho + \gamma_0$ with $\gamma_0\in X_u$ we can apply Donkin's tensor product theorem to get that this is isomorphic to $$\bigoplus_{\gamma\in X_+,\psi\in X_+}t_{\lambda_0}^u(\gamma)t_{\lambda_1}^{r-u}(\psi)T(\gamma + p^u\psi)$$ which immediately gives the claim.
\end{proof}

\subsection{Bounding from below}

The results in the previous section had a very strong requirement on the weight $\lambda$. In this section we will remove this requirement at the cost of changing the equalities to inequalities.

\begin{lm}\label{submoduleinduced}
Let $M$ be a $G$-module satisfying the following for some $\lambda\in X_+$.
\begin{itemize}
\item $\lambda$ is maximal with $M_{\lambda}\neq 0$.
\item $\dim(M_{\lambda}) = 1$.
\item $\soc_{G}(M) = L(\lambda)$.
\end{itemize}
Then there is an injective homomorphism $M\hookrightarrow \nabla(\lambda)$.
\end{lm}

\begin{proof}
By Frobenius reciprocity we have $\Hom_G(M,\nabla(\lambda))\cong \Hom_B(M,\lambda)\neq 0$ since $M$ has a filtration as a $B$-module with $\lambda$ as the top factor, due to the assumption of $\lambda$ being maximal. The image of such a non-zero homomorphism must include $\soc_G(\nabla(\lambda)) = L(\lambda)$ and thus $M_{\lambda}$ is not in the kernel. But then the assumptions show that the socle of $M$ intersects the kernel trivially, and hence that the homomorphism is injective as claimed.
\end{proof}

\begin{prop}\label{tensorproductsubmodule}
Let $\lambda = \lambda_0 + p^r\lambda_1\in X_+$ with $\lambda_0\in X_r$. Then there is an injective homomorphism $\nabla(\lambda_0)\otimes\nabla(\lambda_1)^{(r)}\hookrightarrow \nabla(\lambda)$.
\end{prop}

\begin{proof}
By Lemma \ref{submoduleinduced} it suffices to show that $\soc_G(\nabla(\lambda_0)\otimes\nabla(\lambda_1)^{(r)}) = L(\lambda)$ since the other requirements are clearly satisfied.

To show this we use that $\soc_{G_r}\nabla(\lambda_0) = L(\lambda_0)$ since $\lambda_0\in X_r$ (see \cite[II.3.16]{rags}) and then consider for any $\mu = \mu_0 + p^r\mu_1\in X_+$ with $\mu_0\in X_r$ $$\Hom_G(L(\mu),\nabla(\lambda_0)\otimes\nabla(\lambda_1)^{(r)})$$ $$\cong \Hom_{G/G_r}(L(\mu_1)^{(r)},\Hom_{G_r}(L(\mu_0),\nabla(\lambda_0))\otimes\nabla(\lambda_1)^{(r)}) \cong \begin{cases} k & \mbox{if }\mu = \lambda \\ 0 & \mbox{else}\end{cases}$$ which completes the proof.
\end{proof}

\begin{prop}\label{inductiveinequality}
Assume that $p\geq 2h-2$ and let $\lambda\in \Gamma_r$ with $\lambda = \lambda_0 + p^u\lambda_1$ for some $\lambda_0\in X_u$. Then for any $\nu = \nu_0 + p^u\nu_1\in X_+$ with $\nu_0\in (p^u-1)\rho + X_u$ we have $$t_{\lambda}^r(\nu) \geq \sum_{\sigma\in \Gamma_0}\sum_{\mu \in \Gamma_{r-u}}t_{\lambda_0}^u(\nu_0 + p^u\sigma)t_{\lambda_1}^{r-u}((p^{r-u}-1)\rho + \mu)t_{\sigma,\mu}^{r-u}(\nu_1) \geq t_{\lambda_0}^u(\nu_0)t_{\lambda_1}^{r-u}(\nu_1).$$
\end{prop}

\begin{proof}
By Proposition \ref{tensorproductsubmodule} we have an inclusion $$\St_u\otimes\nabla(\lambda_0)\otimes (\St_{r-u}\otimes\nabla(\lambda))^{(u)}\cong \St_u\otimes\St_{r-u}^{(u)}\otimes\nabla(\lambda_0)\otimes\nabla(\lambda_1)^{(u)}\hookrightarrow \St_u\otimes\St_{r-u}^{(u)}\otimes\nabla(\lambda)\cong \St_r\otimes\nabla(\lambda).$$ We claim that this inclusions splits: Indeed, the inclusion is obtained by tensoring the inclusion $\nabla(\lambda_0)\otimes\nabla(\lambda_1)^{(u)}\hookrightarrow \nabla(\lambda)$ with $\St_r$, and since the highest weights occurring in all of the modules belong to $\Gamma_r$, the resulting short exact sequence consists of tilting modules by Theorem \ref{lambdaalpha0} and hence splits. Hence, for any $\mu\in X_+$, $T(\mu)$ occurs at least as many times in $\St_r\otimes\nabla(\lambda)$ as in $\St_u\otimes\nabla(\lambda_0)\otimes (\St_{r-u}\otimes\nabla(\lambda_1))^{(u)}$. But by definition we have $$\St_u\otimes\nabla(\lambda_0)\otimes (\St_{r-u}\otimes\nabla(\lambda_1))^{(u)}\cong \left(\bigoplus_{\mu\in X_+}t_{\lambda_0}^u(\mu)T(\mu)\right)\otimes\left(\bigoplus_{\gamma\in X_+}t_{\lambda_1}^{r-u}(\gamma)T(\gamma)\right)^{(u)}$$ and the first inequality follows in the same way as 
in the proof of Proposition \ref{inductively}, while the second follows by only considering the summands in the first factor with $\mu\in X_u$.
\end{proof}

We can also get part of the above without the assumption on $p$.

\begin{prop}
Let $\lambda\in \Gamma_r$ with $\lambda = \lambda_0 + p^u\lambda_1$ for some $\lambda_0\in X_u$. Then for any $\nu = \nu_0 + p^u\nu_1\in X_+$ with $\nu_0\in (p^u-1)\rho + X_u$ we have $t_{\lambda}^r(\nu) \geq t_{\lambda_0}^u(\nu_0)t_{\lambda_1}^{r-u}(\nu_1)$.
\end{prop}

\begin{proof}
This follows in the same way as Proposition \ref{inductiveinequality} by using that $T((p^u-1)\rho + \mu + p^u\gamma)$ is a direct summand of $T((p^u-1)\rho + \mu)\otimes T(\gamma)^{(u)}$ for any $\mu\in X_r$ and any $\gamma\in X_+$.
\end{proof}

\section{Reciprocity between tilting modules and simple modules}

In this section we will give a new proof of the reciprocity between simple modules and those tilting modules which are injective as $G_r$-modules. Previous proofs of this result (\cite[Satz 5.9]{jantzen80a}, \cite[Proposition 1.13]{richewilliamson15}) have required that $p\geq 2h-2$, without this being replaceable by Donkin's tilting conjecture. The present proof almost does this, but it does introduce a new assumption which we show cannot be avoided.

\subsection{A bilinear form}

Let $M$ and $N$ be finite dimensional $G$-modules. By \cite[Remark to Lemma II.5.8]{rags} the set of $[\nabla(\lambda)]$ with $\lambda\in X_+$ is a $\Z$-basis of $\ZX$, so we can write $$[M\otimes N^*] = \sum_{\lambda\in X_+}a_{\lambda}[\nabla(\lambda)]$$ for suitable integers $a_{\lambda}$. We define a pairing of finite dimensional $G$-modules by $$\lb M, N\rb = a_0$$ with $a_0$ as in the above sum, which defines a bilinear form on $\ZX$. It has the following basic properties:

\begin{prop}\label{formproperties}
Let $M,N,V$ be finite dimensional $G$-modules.
\begin{enumerate}
\item $\lb\cdot,\cdot\rb$ is symmetric.
\item If $M\otimes N^*$ has a good filtration then $\lb M,N\rb = [M\otimes N^*:\nabla(0)]_{\nabla} = \dim\Hom_G(N,M)$.
\item If $M$ has a good filtration then $\lb M,\nabla(\lambda)\rb = [M:\nabla(\lambda)]_{\nabla}$.
\item The set $\{[\nabla(\lambda)]\mid \lambda\in X_+\}$ is an orthonormal basis of $\ZX$ with respect to $\lb \cdot,\cdot\rb$.
\item $\lb \cdot,\cdot\rb$ is non-degenerate.
\item $\lb M\otimes V,N\rb = \lb M,N\otimes V^*\rb$.
\item $\lb M,N\rb = \lb M^*,N^*\rb$.
\end{enumerate}
\end{prop}

\begin{proof}
$1.$ follows from the fact that $\nabla(0)$ is self-dual together with the fact that if we write $[M] = \sum_{\lambda\in X_+}a_{\lambda}[\nabla(\lambda)]$ then $[M^*] = \sum_{\lambda\in X_+}a_{\lambda}[\nabla(\lambda^*)]$.

$2.$ Follows directly from the definition together with Theorem \ref{cohcrit}.

$3.$ follows from $2$. together with Theorem \ref{cohcrit} by noting that $$\lb M,\nabla(\lambda)\rb = \lb M,\Delta(\lambda)\rb = [M\otimes\nabla(\lambda^*):\nabla(0)]_{\nabla}$$ $$= \dim\Hom_G(\Delta(0),M\otimes\nabla(\lambda^*))= \dim\Hom_G(\Delta(\lambda),M) = [M:\nabla(\lambda)]_{\nabla}.$$

$4.$ follows directly from $3.$ and $5.$ follows directly from $4.$

$6.$ follows directly from the definition and $7.$ is clear from $6.$ together with $1.$
\end{proof}

Note that since the fourth property listed above uniquely characterizes the form, we see that it agrees with the form defined by Donkin in \cite[p. 49]{donkin93}, which also satisfies this property. That this is the case is due to the fact that either form is uniquely determined by the characters of the modules involved, so we can freely exchange any $\nabla(\lambda)$ by $\Delta(\lambda)$, and applying \cite[Proposition II.4.16]{rags} shows that when we apply the form defined by Donkin to a pair $(\Delta(\lambda),\nabla(\mu))$ we get precisely $\delta_{\lambda,\mu}$.

This identifies the form with the Euler characteristic, which may in some cases make the following results seem more natural.

\subsection{Computing the form}

In order to prove the reciprocity between tilting modules and simple modules, we will note that by Proposition \ref{formproperties} for any $\lambda,\mu\in X_+$ we have $[T(\lambda):\nabla(\mu)]_{\nabla} = \lb T(\lambda),\nabla(\mu)\rb$ and if we write $[\nabla(\mu)] = \sum_{\nu\in X_+}b_{\nu}[L(\nu)]$ then $$\lb T(\lambda),\nabla(\mu)\rb = \sum_{\nu\in X_+}b_{\nu}\lb T(\lambda),L(\nu)\rb$$ so we need to be able to compute these $\lb T(\lambda),L(\nu)\rb$ for suitable $\lambda,\nu\in X_+$.

The relevant highest weights for the tilting modules in question will all be of the form $(p^r-1)\rho + \mu$ for some $\mu\in X_+$, so for convenience we will adopt the notation $\wh\lambda = 2(p^r-1)\rho - \lambda^*$ for $\lambda\in X_r$ which is chosen such that if we assume that $T(\wh\lambda)$ is indecomposable as a $G_r$-module then $\soc_GT(\wh\lambda) = L(\lambda)$. Note that $\wh\lambda$ depends on $r$ even though this is not apparent in the notation, but since we will not be varying $r$ in this section, it should not cause any problems (in the notation previously introduced, this could also be written as $\wh\lambda = (p^r-1)\rho + w_r(\lambda)$ but this would be more cumbersome).

We start with a few lemmas.

\begin{lm}\label{formwithpr}
Let $\lambda,\nu\in X_r$, $\sigma,\mu\in X_+$ and assume that $T(\wh\lambda)$ is indecomposable as a $G_r$-module and that $\St_r\otimes L(\nu)$ has a good filtration. Then $$\lb T(\wh\lambda + p^r\sigma), L(\nu)\otimes\Delta(\mu)^{(r)}\rb = \begin{cases} [T(\sigma):\nabla(\mu)]_{\nabla} & \mbox{if }\nu = \lambda \\ 0 & \mbox{else}\end{cases}.$$
\end{lm}

\begin{proof}
By Proposition \ref{tensorproductsteinberggood} and the Andersen--Haboush tensor product theorem we see that $\St_r\otimes L(\nu^*)\otimes\nabla(\mu^*)^{(r)}$ has a good filtration, so by Theorem \ref{goodfilttensorwithtilting} we can apply Proposition \ref{formproperties}(2) and Donkin's tensor product theorem to get $$\lb T(\wh\lambda + p^r\sigma),L(\nu)\otimes\Delta(\mu)^{(r)}\rb = \dim\Hom_G(L(\nu)\otimes\Delta(\mu)^{(r)},T(\wh\lambda)\otimes T(\sigma)^{(r)})$$ $$= \dim\Hom_{G/G_r}(\Delta(\mu)^{(r)},\Hom_{G_r}(L(\nu),T(\wh\lambda))\otimes T(\sigma)^{(r)})$$ $$= \begin{cases}\dim\Hom_G(\Delta(\mu),T(\sigma)) & \mbox{if }\nu = \lambda \\ 0 & \mbox{else}\end{cases}$$ where the last equality uses the assumption that $T(\wh\lambda)$ is indecomposable as a $G_r$-module and thus has 
socle equal to $L(\lambda)$. The claim now follows from Theorem \ref{cohcrit}.
\end{proof}

\begin{rem}
We have attempted to make the assumptions in the above lemma as well as in the remaining result in this section as weak as possible, but it is worth noting that the conditions requiring certain tilting modules to be indecomposable as $G_r$-modules as well as the ones requiring certain tensor products between a Steinberg module and a simple module to have a good filtration can be replaced by Donkin's tilting conjecture, since this also implies that $\St_r\otimes L(\mu)$ has a good filtration for all $\mu\in X_r$ by \cite[Theorem 9.4.1]{kildetoftnakano15}.
\end{rem}

\begin{lm}\label{formwithwrongrestricted}
Let $\lambda,\nu\in X_r$, $\sigma,\mu\in X_+$ and assume that $T(\wh\lambda)$ is indecomposable as a $G_r$-module and that $\St_r\otimes L(\nu)$ has a good filtration. If $\nu\neq \lambda$ then $$\lb T(\wh\lambda + p^r\sigma),L(\nu + p^r\mu)\rb = 0.$$
\end{lm}

\begin{proof}
Assume that $\nu\neq \lambda$ and assume for the purposes of induction that for all $\gamma\in X_+$ with $\gamma < \mu$ we have $\lb T(\wh\lambda + p^r\sigma),L(\nu + p^r\gamma)\rb = 0$. Write $$[\Delta(\mu)] = [L(\mu)] + \sum_{\gamma\in X_+,\, \gamma < \mu}b_{\gamma}[L(\gamma)].$$

Then by Steinberg's tensor product theorem $$\lb T(\wh\lambda + p^r\sigma),L(\nu + p^r\mu)\rb = \lb T(\wh\lambda + p^r\sigma),L(\nu)\otimes\Delta(\mu)^{(r)}\rb - \sum_{\gamma\in X_+,\,\gamma < \mu}b_{\gamma}\lb T(\wh\lambda + p^r\sigma),L(\nu + p^r\gamma)\rb$$ but whenever $b_{\gamma}\neq 0$ we have $\gamma < \mu$ so by assumption the sum is $0$. Also, by Lemma \ref{formwithpr} the first term is $0$, so this shows that $\lb T(\wh\lambda),L(\nu + p^r\mu)\rb = 0$ by induction, where the base case follows from Lemma \ref{formwithpr}.
\end{proof}

In the following lemma, we need to assume that $\sigma$ is not strongly linked to $\mu$. For more information on strong linkage, we refer to \cite[II.6]{rags}.

\begin{lm}\label{formwithnotlinkedto0}
Let $\lambda,\nu\in X_r$, $\sigma,\mu\in X_+$ and assume that $T(\wh\lambda)$ is indecomposable as a $G_r$-module and that $\St_r\otimes L(\nu)$ has a good filtration. Assume further that $\nabla(\sigma)$ is simple. Let $\mu\in X_+$ and assume that $\sigma$ is not strongly linked to $\mu$. Then $\lb T(\wh\lambda + p^r\sigma),L(\nu + p^r\mu)\rb = 0$.
\end{lm}

\begin{proof}
Write $$[\Delta(\mu)] = [L(\mu)] + \sum_{\gamma \in X_+,\, \gamma < \mu}b_{\gamma}[L(\gamma)].$$ By Steinberg's tensor product theorem we have $$\lb T(\wh\lambda + p^r\sigma),L(\nu + p^r\mu)\rb = \lb T(\wh\lambda + p^r\sigma),L(\nu)\otimes\Delta(\mu)^{(r)}\rb - \sum_{\gamma\in X_+,\,\gamma < \mu}b_{\gamma}\lb T(\wh\lambda + p^r\sigma),L(\nu + p^r\gamma)\rb$$ and $\sigma$ is not strongly linked to any $\gamma$ with $b_{\gamma}\neq 0$ since these $\gamma$ are all strongly linked to $\mu$. But now we can assume by induction that all terms in the sum are $0$, so it remains to show that the first term is $0$. By Lemma \ref{formwithpr} this is either $0$ or $[T(\sigma):\nabla(\mu)]_{\nabla}$, and the latter equals $[\nabla(\sigma):\nabla(\mu)]_{\nabla} = 0$ since $\nabla(\sigma)$ was assumed to be simple and $\sigma$ was not strongly linked to $\mu$ (so in particular, we have $\sigma\neq\mu$).
\end{proof}

\begin{thm}\label{formwithsimple}
Let $\lambda,\nu\in X_r$, $\sigma,\mu\in X_+$ and assume that $T(\wh\lambda)$ is indecomposable as a $G_r$-module and $\nabla(\sigma)$ is simple. Assume further that $\St_r\otimes L(\nu)$ has a good filtration and that either $\nabla(\mu)$ is simple or $\sigma$ is not strongly linked to $\mu$.

Then $$\lb T(\wh\lambda + p^r\sigma),L(\nu + p^r\mu)\rb = \begin{cases}1 & \mbox{if }\nu + p^r\mu = \lambda + p^r\sigma \\ 0 & \mbox{else}\end{cases}.$$
\end{thm}

\begin{proof}
If $\lb T(\wh\lambda + p^r\sigma),L(\nu + p^r\mu)\rb \neq 0$ then by Lemma \ref{formwithwrongrestricted} and Lemma \ref{formwithnotlinkedto0} we have $\nu = \lambda$ and $\sigma$ is strongly linked to $\mu$ and thus by assumption we must have that $\nabla(\mu)$ is simple. But since we also have that $T(\sigma)\cong\nabla(\sigma)$ the claim now follows from Lemma \ref{formwithpr}.
\end{proof}

We also include the following which exchanges the conditions on $\sigma$ for stronger conditions on $\mu$.

\begin{prop}
Let $\lambda,\nu\in X_r$, $\sigma,\mu\in X_+$ and assume that $T(\wh\lambda)$ is indecomposable as a $G_r$-module and that $\St_r\otimes L(\nu)$ has a good filtration. Assume further that whenever $\psi\in X_+$ with $\psi\leq \mu$ then $\psi\not\leq\sigma$. Then $$\lb T(\wh\lambda + p^r\sigma),L(\nu + p^r\mu)\rb = 0.$$
\end{prop}

\begin{proof}
Since the condition on $\mu$ is inherited by any $\gamma$ with $\gamma\leq \mu$, this follows in the same way as Lemma \ref{formwithnotlinkedto0}.
\end{proof}

\subsection{Reciprocity}

We can now prove the reciprocity between tilting modules and simple modules. Note that the result is in particular applicable whenever $\lambda + p^r\sigma\in\Gamma_r$ and $\mu\in\Gamma_r$, as long as we assume Donkin's tilting conjecture (so if $p\geq 2h-2$ it includes \cite[Satz 5.9]{jantzen80a} and \cite[Proposition 1.13]{richewilliamson15}). Also, as mentioned earlier, if we assume Donkin's tilting conjecture then the conditions on $\lambda$ and $\nu_0$ will be automatic, whereas the condition on $\nu_1$ will not (as will be seen in a later example).

\begin{cor}\label{reciprocity}
Let $\lambda\in X_r$ and $\sigma,\mu\in X_+$ and assume that $T(\wh\lambda)$ is indecomposable as a $G_r$-module and $\nabla(\sigma)$ is simple. Assume further that whenever $L(\nu_0 + p^r\nu_1)$ is a composition factor of $\nabla(\mu)$ with $\nu_0\in X_r$ then $\St_r\otimes L(\nu_0)$ has a good filtration and if $\nabla(\nu_1)$ is not simple then $\sigma$ is not strongly linked to $\nu_1$. Then $$[T(\wh\lambda + p^r\sigma):\nabla(\mu)]_{\nabla} = [\nabla(\mu):L(\lambda + p^r\sigma)]_G.$$
\end{cor}

\begin{proof}
Since by Proposition \ref{formproperties}(3) we have $[T(\wh\lambda + p^r\sigma):\nabla(\mu)]_{\nabla} = \lb T(\wh\lambda + p^r\sigma),\nabla(\mu)\rb$ the claim follows from Theorem \ref{formwithsimple} and Lemma \ref{formwithpr}.
\end{proof}

\subsection{A counterexample}

We would like to be able to get rid of the requirement on the composition factors of $\nabla(\mu)$ in Corollary \ref{reciprocity}, or at least lessen them to requiring that $\mu\in X_r$, but this is unfortunately not possible as we will now show.

\begin{prop}\label{formwithbadweight}
Assume that $p\geq h$. Let $\lambda\in X_r$ and assume that $T(\wh\lambda)$ is indecomposable as a $G_r$-module and that $\St_r\otimes L(\lambda)$ has a good filtration. Then $$\lb T(\wh\lambda),L(\lambda + p^r(p-h+1)\alpha_0)\rb = -1.$$
\end{prop}

\begin{proof}
First observe that by the Jantzen sum formula (\cite[Proposition II.8.19]{rags}) we have $[\Delta((p-h+1)\alpha_0)] = [L((p-h+1)\alpha_0)] + [L(0)]$ so by Steinberg's tensor product theorem we have $[L(\lambda + p^r(p-h+1)\alpha_0)] = [L(\lambda)\otimes\Delta((p-h+1)\alpha_0)^{(r)}] - [L(\lambda)]$.

The claim now follows from Lemma \ref{formwithpr} and Theorem \ref{formwithsimple}.
\end{proof}

\begin{ex}
Let $G = SL_5$ and $p=5$. Let $\mu = (2,3,3,2)$ (we will write all weights in terms of the fundamental weights so $(a_1,a_2,a_3,a_4) = a_1\omega_1 + a_2\omega_2 + a_3\omega_3 + a_4\omega_4$ or in other words, if we label the simple roots $\alpha_1,\alpha_2,\alpha_3,\alpha_4$ going from one end of the Dynkin diagram to the other, this is the unique weight $\lambda$ such that $\langle\lambda,\alpha_i^{\vee}\rangle = a_i$ for $i=1,2,3,4$). Then $L(p(p-h+1)\alpha_0) = L(5,0,0,5)$ is a composition factor of $\nabla(\mu)$ with multiplicity $1$ (as can be checked by applying the Jantzen sum formula) and thus by Proposition \ref{formwithbadweight} and arguing as above (with the same assumptions as previously) we have $$[T(\wh 0):\nabla(\mu)]_{\nabla} = [\nabla(\mu):L(0)]_G - 1$$ which is not equal to $[\nabla(\mu):L(0)]_G$.
\end{ex}

Throughout this section we have had result needing the assumption that $T(\wh\lambda)$ was indecomposable as a $G_r$-module for some $\lambda\in X_r$. This condition is in fact also necessary if we are to obtain results like these. We will postpone the proof of this to later, as it will require us to know more about the dimension of homomorphism spaces for $G_r$.

\section{Restriction to a Frobenius kernel}

In this section, we will consider the restrictions of $T((p^r-1)\rho + \lambda)$ and $\St_r\otimes\nabla(\lambda)$ to $G_r$.

\subsection{Relating $\Hom$-spaces for $G_r$ and $G$}

We will need the following result which follows from the discussion in \cite[5.1]{bnppss12}.

\begin{thm}\label{grext}
Let $M$ and $N$ be finite dimensional $G$-modules. There exists a finite dimensional $G$-module $\Q_{M,N}^{r}$ such that.
\begin{itemize}
\item $\Hom_{G_r}(M,N)\cong \Hom_G(M,N\otimes\Q_{M,N}^{r})$.
\item $\Q_{M,N}^{r}$ has a filtration with factors of the form $\nabla(\lambda)^{(r)}$, each occurring with multiplicity either $0$ or $\dim\nabla(\lambda)$.
\item If $\nabla(\lambda)^{(r)}$ does not occur in $\Q_{M,N}^{r}$ then $\Hom_G(M,N\otimes\nabla(\lambda)^{(r)}) = 0$.
\end{itemize}
\end{thm}

Note that while the module $\Q_{M,N}^{r}$ above is not uniquely determined by the given conditions, this will not be a problem, as we will only need it to compute certain $\Hom$-spaces.

\subsection{Dimension of $\Hom$-spaces for $G_r$}

Using Theorem \ref{grext}, we can determine the dimension of the space of homomorphisms between certain tilting modules and costandard modules as $G_r$-modules.

\begin{thm}\label{dimhomgr}
Let $\lambda,\nu\in X_+$. Then $$\dim\Hom_{G_r}(T((p^r-1)\rho + \lambda),\nabla(\nu)) = \sum_{\mu\in X_+}\dim(\nabla(\mu))[T((p^r-1)\rho + \lambda^*)\otimes\nabla(\mu)^{(r)}:\nabla(\nu^*)]_{\nabla}.$$ In particular, $$\dim\Hom_{G_r}(\St_r,\nabla(\nu)) = \sum_{\mu\in X_+}\dim(\nabla(\mu))[\nabla((p^r-1)\rho + p^r\mu):\nabla(\nu^*)]_{\nabla}$$ $$= \sum_{\mu\in X_+}\dim(\nabla(\mu))[\nabla(\nu):\nabla((p^r-1)\rho + p^r\mu)]_{\nabla} = \begin{cases}\dim(\nabla(\mu)) & \mbox{if }\nu = (p^r-1)\rho + p^r\mu \\ 0 & \mbox{else}\end{cases}.$$
\end{thm}

\begin{proof}
By Theorem \ref{grext} there is a finite dimensional $G$-module $\Q = \Q_{T((p^r-1)\rho + \lambda),\nabla(\nu)}^r$ such that $\Hom_{G_r}(T((p^r-1)\rho + \lambda),\nabla(\nu)) \cong \Hom_G(T((p^r-1)\rho + \lambda),\nabla(\nu)\otimes \Q)$. 

We first note that by Corollary \ref{exttensor} for all $i\geq 1$ and all $\mu\in X_+$ we have $\Ext_G^i(T((p^r-1)\rho + \lambda),\nabla(\nu)\otimes \nabla(\mu)^{(r)}) = 0$.

This means that since $\Q$ has a filtration with factors of the form $\nabla(\mu)^{(r)}$ by Theorem \ref{grext} we get $$\dim\Hom_{G_r}(T((p^r-1)\rho + \lambda),\nabla(\nu)) = \dim\Hom_G(T((p^r-1)\rho + \lambda),\nabla(\nu)\otimes \Q)$$ $$= \sum_{\mu\in X_+}\dim(\nabla(\mu))\dim\Hom_G(T((p^r-1)\rho + \lambda),\nabla(\nu)\otimes\nabla(\mu)^{(r)})$$ where we can sum over all $\mu\in X_+$ by the third property of $\Q$ listed in Theorem \ref{grext}.

The first claim now follows from Corollary \ref{exttensor} while the first equality of the second claim follows from this by applying the Andersen--Haboush tensor product theorem, the second follows by similar arguments after changing the summation to be over $\mu^*$ instead of $\mu$, and the final equality is clear.
\end{proof}

Using the bilinear form introduced earlier, we can extend this to be valid for any finite dimensional $G$-module.

\begin{thm}\label{dimhomgrall}
Let $\lambda\in X_+$ and $M$ be a finite dimensional $G$-module. Then $$\dim\Hom_{G_r}(T((p^r-1)\rho + \lambda),M) = \sum_{\mu\in X_+}\dim(\nabla(\mu))\lb T((p^r-1)\rho + \lambda)\otimes\nabla(\mu)^{(r)},M\rb.$$ In particular, $$\dim\Hom_{G_r}(\St_r,M) = \sum_{\mu\in X_+}\dim(\nabla(\mu))\lb \nabla((p^r-1)\rho + p^r\mu),M\rb.$$
\end{thm}

\begin{proof}
Since $T((p^r-1)\rho + \lambda)$ is projective as a $G_r$-module, the map $M\mapsto \dim\Hom_{G_r}(T((p^r-1)\rho + \lambda),M)$ defines a linear map from $\ZX$ to $\Z$.

Similarly, the map $M\mapsto \sum_{\mu\in X_+}\dim(\nabla(\mu))\lb T((p^r-1)\rho + \lambda^*)\otimes\nabla(\mu)^{(r)},M^*\rb$ is linear.

By Theorem \ref{dimhomgr} and Proposition \ref{formproperties} these maps agree on $\nabla(\nu)$ for all $\nu\in X_+$, since $\nabla(\nu^*)$ and $\nabla(\nu)^*$ have the same character. But then they agree on all finite dimensional $G$-modules by \cite[Remark to Lemma II.5.8]{rags}, and the first claim follows by using Proposition \ref{formproperties}(7) and changing the summation to be over $\mu^*$.

The second claim follows by applying the Andersen--Haboush tensor product theorem.
\end{proof}

\subsection{Relation to Donkin's tilting conjecture}

We can now prove that in a suitable formulation, the reciprocity between tilting modules and simple modules in fact implies Donkin's tilting conjecture.

For another recent result with the same conclusion, see \cite[Theorem 4.3.1]{sobaje16}.

\begin{thm}\label{impliesdonkin}
Let $\lambda\in X_r$ and assume that for $\nu\in X_r$ and $\mu\in X_+$ we have $$\lb T(\wh\lambda),L(\nu)\otimes\Delta(\mu)^{(r)}\rb = \begin{cases}1 & \mbox{if }\nu = \lambda\mbox{ and }\mu = 0 \\ 0 & \mbox{else}\end{cases}.$$ Then $T(\wh\lambda)$ is indecomposable as a $G_r$-module.
\end{thm}

\begin{proof}
By Theorem \ref{dimhomgrall} and Proposition \ref{formproperties}(7) the assumptions imply that for $\nu\in X_r$ we have $$\dim\Hom_{G_r}(T(\wh\lambda),L(\nu)) = \begin{cases}1 & \mbox{if }\nu = \lambda \\ 0 & \mbox{else}\end{cases}$$ which gives the claim.
\end{proof}

\begin{cor}
The following are equivalent.
\begin{enumerate}
\item Donkin's tilting conjecture.
\item $\lb T(\wh\lambda),L(\nu)\otimes\Delta(\mu)^{(r)}\rb = \delta_{\nu,\lambda}\delta_{\mu,0}$ for all $\lambda,\nu\in X_r$ and all $\mu\in X_+$.
\end{enumerate}
\end{cor}

\begin{proof}
This is the combination of Lemma \ref{formwithpr} and Theorem \ref{impliesdonkin}.
\end{proof}

\subsection{Decomposition of $\St_r\otimes\nabla(\lambda)$ for $G_r$}

Since $\St_r\otimes\nabla(\lambda)$ is injective as a $G_r$-module, for any $\lambda\in X_+$ we can write $$\St_r\otimes\nabla(\lambda) \cong \bigoplus_{\nu\in X_r}d_{\lambda}^r(\nu)Q_r(\nu)$$ where $d_{\lambda}^r(\nu) = \dim\Hom(L(\nu),\St_r\otimes\nabla(\lambda))$.

We can then compute the $d_{\lambda}^r$ using the previous results, giving a formula in terms of standard character data for $G$.

\begin{prop}
For any $\lambda\in X_+$ and any $\nu\in X_r$ we have $$d_{\lambda}^r(\nu) = \sum_{\mu\in X_+}\dim(\nabla(\mu))\lb \nabla((p^r-1)\rho + p^r\mu),L(\nu^*)\otimes\nabla(\lambda)\rb.$$
\end{prop}

\begin{proof}
This follows straight from Theorem \ref{dimhomgrall}.
\end{proof}

\section{Restriction to a finite Chevalley group}

In this section, we consider what happens when $T((p^r-1)\rho + \lambda)$ and $\St_r\otimes \nabla(\lambda)$ are restricted to the finite Chevalley group $\Gq$. Recall that $q = p^r$.

\subsection{Relating $\Hom$-spaces for $\Gq$ and $G$}

We will need the following result which follows directly from the proof of \cite[Theorem 3.2.1]{bnppss12}.

\begin{thm}\label{gqext}
Let $M$ and $N$ be finite dimensional $G$-modules. There exists a finite dimensional $G$-module $\G_{M,N}^{r}$ such that
\begin{itemize}
\item $\Hom_{\Gq}(M,N)\cong \Hom_G(M,N\otimes\G_{M,N}^{r})$.
\item $\G_{M,N}^{r}$ has a filtration with factors of the form $\nabla(\lambda^*)\otimes \nabla(\lambda)^{(r)}$, each occurring with multiplicity at most $1$.
\item If $\nabla(\lambda^*)\otimes\nabla(\lambda)^{(r)}$ does not occur in $\G_{M,N}^{r}$ then $\Hom_G(M,N\otimes\nabla(\lambda^*)\otimes\nabla(\lambda)^{(r)}) = 0$.
\end{itemize}
\end{thm}

Note that while the module $\G_{M,N}^{r}$ is not uniquely determined by the conditions in the above theorem, this will not be a problem here as we will only be using it for computing various $\Hom$-spaces.

\subsection{Dimensions of $\Hom$-spaces for $\Gq$}

The $\lambda = 0$ case of the following is a special case of \cite[Proposition 2.5]{wanwang11} (whose proof is due to C. Pillen). The proof is very similar, using some of the previous results from this paper to make it applicable to $\lambda\neq 0$.

\begin{thm}\label{dimhomnabla}
Let $\lambda,\nu\in X_+$. Then $$\dim\Hom_{\Gq}(T((p^r-1)\rho + \lambda),\nabla(\nu)) = \sum_{\mu\in X_+}[T((p^r-1)\rho + \lambda^*)\otimes\nabla(\mu)^{(r)}\otimes \nabla(\nu):\nabla(\mu)]_{\nabla}.$$ In particular, $$\dim\Hom_{G(\F_q)}(\St_r,\nabla(\nu)) = \sum_{\mu\in X_+}[\nabla((p^r-1)\rho + p^r\mu)\otimes \nabla(\nu):\nabla(\mu)]_{\nabla}$$ $$= \sum_{\mu\in X_+}[\nabla(\mu)\otimes\nabla(\nu):\nabla((p^r-1)\rho + p^r\mu)]_{\nabla}.$$
\end{thm}

\begin{proof}
By Theorem \ref{gqext} there is a finite dimensional $G$-module $\G = \G_{T((p^r-1)\rho + \lambda),\nabla(\nu)}^{r}$ such that $\Hom_{\Gq}(T((p^r-1)\rho + \lambda),\nabla(\nu)) \cong \Hom_G(T((p^r-1)\rho + \lambda),\nabla(\nu)\otimes\G)$.

First we claim that for all $i\geq 1$ and all $\mu\in X_+$ we have $\Ext_{G}^i(T((p^r-1)\rho + \lambda),\nabla(\nu)\otimes \nabla(\mu^*)\otimes\nabla(\mu)^{(r)}) = 0$. This follows from Corollary \ref{exttensor} since $\St_r\otimes\nabla(\nu)$ has a good filtration by Theorem \ref{tensorproductgoodfilt}.

This means that since $\G$ has a filtration with factors of the form $\nabla(\mu^*)\otimes\nabla(\mu)^{(r)}$ by Theorem \ref{gqext} we have $$\dim\Hom_{\Gq}(T((p^r-1)\rho + \lambda),\nabla(\nu)) = \dim\Hom_G(T((p^r-1)\rho + \lambda),\nabla(\nu)\otimes \G)$$ $$=\sum_{\mu\in X_+}\dim\Hom_G(T((p^r-1)\rho + \lambda),\nabla(\nu)\otimes\nabla(\mu^*)\otimes\nabla(\mu)^{(r)})$$ where we can sum over all $\mu\in X_+$ due to the third property of $\G$ listed in Theorem \ref{gqext}. The first claim now follows directly from Corollary \ref{exttensor}.

The first equality of the second claim follows by applying the Andersen--Haboush tensor product theorem, while the second follows by a similar argument to the above once the summation is changed to be over $\mu^*$ rather than $\mu$.
\end{proof}

Using the bilinear form introduced previously, we can in fact extend the above result to hold for all finite dimensional $G$-modules, rather than just $\nabla(\nu)$. It is interesting to compare the $\lambda = 0$ case to \cite[Theorem 6.9]{humphreysfinlie} which is essentially the same formula, but with characters of costandard modules replaced by characters of simple modules.

It is also interesting to compare the below result to \cite[Theorem 10.11]{humphreysfinlie}. When $\lambda\in X_r$ is such that $T((p^r-1)\rho)+\lambda)$ is indecomposable as a $G_r$-module and $M = L(\gamma)$ for some $\gamma\in X_r$, the formulas will give the same results, once the weights have been relabeled suitably. And in fact, the formulas producing the same results for all weights is equivalent to $T((p^r-1)\rho + \lambda)$ being indecomposable as a $G_r$-module.

\begin{thm}\label{dimhomall}
Let $M$ be a finite dimensional $G$-module and $\lambda\in X_+$. Then $$\dim\Hom_{\Gq}(T((p^r-1)\rho + \lambda),M) = \sum_{\mu\in X_+}\lb T((p^r-1)\rho + \lambda^*)\otimes\nabla(\mu)^{(r)}\otimes M,\nabla(\mu)\rb$$ $$= \sum_{\mu\in X_+}\lb T((p^r-1)\rho + \lambda)\otimes \nabla(\mu)^{(r)}\otimes\nabla(\mu^*),M\rb.$$ In particular, $$\dim\Hom_{G(\F_q)}(\St_r,M) = \sum_{\mu\in X_+}\lb \nabla((p^r-1)\rho + p^r\mu)\otimes M,\nabla(\mu)\rb$$ $$= \sum_{\mu\in X_+}\lb \nabla(\mu)\otimes M,\nabla((p^r-1)\rho + p^r\mu)\rb.$$
\end{thm}

\begin{proof}
Since $T((p^r-1)\rho + \lambda)$ is projective as a $\Gq$-module, the map $M\mapsto \dim\Hom_{\Gq}(T((p^r-1)\rho + \lambda),M)$ defines a $\mathbb{Z}$-linear map from $\ZX$ to $\mathbb{Z}$.

Similarly, the map $M\mapsto \sum_{\mu\in X_+}\lb T((p^r-1)\rho + \lambda^*)\otimes\nabla(\mu)^{(r)}\otimes M,\nabla(\mu)\rb$ is $\mathbb{Z}$-linear.

By Theorem \ref{dimhomnabla} together with Proposition \ref{formproperties}(3) the above two $\mathbb{Z}$-linear maps agree on $\nabla(\nu)$ for all $\nu\in X_+$ and since these form a $\mathbb{Z}$-basis of $\ZX$ by \cite[Remark to Lemma II.5.8]{rags}, the maps must agree on all finite dimensional $G$-modules, which yields the first claim.

The second equality follows by applying Proposition \ref{formproperties}.
\end{proof}

\subsection{Decomposition of $\St_r\otimes \nabla(\lambda)$ for $\Gq$}
Since $\St_r\otimes \nabla(\lambda)$ is projective as a $\Gq$-module, we can write $$\St_r\otimes\nabla(\lambda) \cong \bigoplus_{\nu\in X_r}p_{\lambda}^r(\nu)P_r(\nu)$$ where $p_{\lambda}^r(\nu) = \dim\Hom_{\Gq}(L(\nu),\St_r\otimes\nabla(\lambda))$.

Using the previous results, we can then give a formula for the $p_{\lambda}^r(\nu)$ in terms of standard character data for $G$.

\begin{cor}
Let $\lambda\in X_+$ and $\nu\in X_r$. Then $$p_{\lambda}^r(\nu) = \sum_{\mu\in X_+}\lb \nabla(\mu)\otimes\nabla(\lambda)\otimes L(\nu^*),\nabla((p^r-1)\rho + p^r\mu)\rb.$$
\end{cor}

\begin{proof}
This follows directly from Theorem \ref{dimhomall}.
\end{proof}

\bibliography{bibtex-full}

\newcommand{\etalchar}[1]{$^{#1}$}
\def\cprime{$'$}
\begin{thebibliography}{BNP{\etalchar{+}}12}

\bibitem[And80]{andersen80}
Henning~Haahr Andersen.
\newblock The {F}robenius morphism on the cohomology of homogeneous vector
  bundles on {$G/B$}.
\newblock {\em Ann. of Math. (2)}, 112(1):113--121, 1980.

\bibitem[And01]{andersen01}
Henning~Haahr Andersen.
\newblock {$p$}-filtrations and the {S}teinberg module.
\newblock {\em J. Algebra}, 244(2):664--683, 2001.

\bibitem[BDM11]{bdm11}
C.~Bowman, S.~R. Doty, and S.~Martin.
\newblock Decomposition of tensor products of modular irreducible
  representations for {${\rm SL}_3$}.
\newblock {\em Int. Electron. J. Algebra}, 9:177--219, 2011.
\newblock With an appendix by C. M. Ringel.

\bibitem[BDM15]{bdm15}
C.~Bowman, S.~R. Doty, and S.~Martin.
\newblock Decomposition of tensor products of modular irreducible
  representations for {${\rm SL}_3$}: the {$p\geq 5$} case.
\newblock {\em Int. Electron. J. Algebra}, 17:105--138, 2015.

\bibitem[BNP{\etalchar{+}}12]{bnppss12}
Christopher~P. Bendel, Daniel~K. Nakano, Brian~J. Parshall, Cornelius Pillen,
  Leonard~L. Scott, and David~I. Stewart.
\newblock Bounding extensions for finite groups and {F}robenius kernels.
\newblock {\em arXiv:1208.6333 [math.RT]}, 2012.

\bibitem[DH05]{dotyhenke05}
Stephen Doty and Anne Henke.
\newblock Decomposition of tensor products of modular irreducibles for {${\rm
  SL}_2$}.
\newblock {\em Q. J. Math.}, 56(2):189--207, 2005.

\bibitem[Don81]{donkin81}
Stephen Donkin.
\newblock A filtration for rational modules.
\newblock {\em Math. Z.}, 177(1):1--8, 1981.

\bibitem[Don93]{donkin93}
Stephen Donkin.
\newblock On tilting modules for algebraic groups.
\newblock {\em Math. Z.}, 212(1):39--60, 1993.

\bibitem[Dru13]{drupieski13}
Christopher~M. Drupieski.
\newblock On projective modules for {F}robenius kernels and finite {C}hevalley
  groups.
\newblock {\em Bull. Lond. Math. Soc.}, 45(4):715--720, 2013.

\bibitem[Hab80]{haboush80}
W.~J. Haboush.
\newblock A short proof of the {K}empf vanishing theorem.
\newblock {\em Invent. Math.}, 56(2):109--112, 1980.

\bibitem[Hum76]{humphreyschevalley}
James~E. Humphreys.
\newblock {\em Ordinary and modular representations of {C}hevalley groups}.
\newblock Lecture Notes in Mathematics, Vol. 528. Springer-Verlag, Berlin-New
  York, 1976.

\bibitem[Hum06]{humphreysfinlie}
James~E. Humphreys.
\newblock {\em Modular representations of finite groups of {L}ie type}, volume
  326 of {\em London Mathematical Society Lecture Note Series}.
\newblock Cambridge University Press, Cambridge, 2006.

\bibitem[Jan80]{jantzen80a}
Jens~C. Jantzen.
\newblock Darstellungen halbeinfacher {G}ruppen und ihrer {F}robenius-{K}erne.
\newblock {\em J. Reine Angew. Math.}, 317:157--199, 1980.

\bibitem[Jan03]{rags}
Jens~Carsten Jantzen.
\newblock {\em Representations of algebraic groups}, volume 107 of {\em
  Mathematical Surveys and Monographs}.
\newblock American Mathematical Society, Providence, RI, second edition, 2003.

\bibitem[KN15]{kildetoftnakano15}
Tobias Kildetoft and Daniel~K. Nakano.
\newblock On good {$(p,r)$}-filtrations for rational {$G$}-modules.
\newblock {\em J. Algebra}, 423:702--725, 2015.

\bibitem[Mat90]{mathieu90}
Olivier Mathieu.
\newblock Filtrations of {$G$}-modules.
\newblock {\em Ann. Sci. \'Ecole Norm. Sup. (4)}, 23(4):625--644, 1990.

\bibitem[RW15]{richewilliamson15}
Simon Riche and Geordie Williamson.
\newblock Tilting modules and the $p$-canonical basis.
\newblock {\em arXiv:1512.08296 [math.RT]}, 2015.

\bibitem[Sob16]{sobaje16}
Paul Sobaje.
\newblock Varieties of ${G}_r$-summands in rational ${G}$-modules.
\newblock {\em arXiv:1605.06330 [math.RT]}, 2016.

\bibitem[Ste63]{steinberg63}
Robert Steinberg.
\newblock Representations of algebraic groups.
\newblock {\em Nagoya Math. J.}, 22:33--56, 1963.

\bibitem[WW11]{wanwang11}
Jinkui Wan and Weiqiang Wang.
\newblock The {${\rm GL}_n(q)$}-module structure of the symmetric algebra
  around the {S}teinberg module.
\newblock {\em Adv. Math.}, 227(4):1562--1584, 2011.

\end{thebibliography}
\bibliographystyle{alpha}

\end{document}